\documentclass[11pt]{amsart}

\usepackage{enumerate,url,amssymb,mathrsfs,upref,pdfsync}
\usepackage{amsmath, amsfonts,amsthm,times,graphics,color}
\usepackage{hyperref}
\newtheorem{theorem}{Theorem}[section]
\newtheorem{lemma}[theorem]{Lemma}

\newtheorem{proposition}[theorem]{Proposition}
\newtheorem{conjecture}[theorem]{Conjecture}
\newtheorem*{proposition*}{Proposition}

\theoremstyle{definition}

\newtheorem{corollary}[theorem]{Corollary}

\theoremstyle{remark}
\newtheorem{remark}[theorem]{Remark}

\numberwithin{equation}{section}

\newcommand{\abs}[1]{\lvert#1\rvert}

\newcommand{\A}{\mathbb{A}}

\newcommand{\C}{\mathbb{C}}

\newcommand{\E}{\mathcal{E}}

\newcommand{\X}{\mathbb{X}}
\newcommand{\Y}{\mathbb{Y}}

\newcommand{\onto}{\overset{{}_{\textnormal{\tiny{onto}}}}{\longrightarrow}}

\DeclareMathOperator{\Mod}{Mod}

\DeclareMathOperator{\id}{id}

\def\XXint#1#2#3{{\setbox0=\hbox{$#1{#2#3}{\int}$}
\vcenter{\hbox{$#2#3$}}\kern-.5\wd0}}

\def\le{\leqslant}
\def\ge{\geqslant}

\begin{document}

\title{Minimisers and Kellogg's theorem}  \subjclass{Primary 31A05;
Secondary 	49Q05}


\keywords{Minimizers, Kellogg theorem, Minimal surfaces,  Annuli}
\author{David Kalaj}
\address{University of Montenegro, Faculty of Natural Sciences and
Mathematics, Cetinjski put b.b. 81000 Podgorica, Montenegro}
\email{davidk@ucg.ac.me}

\author{Bernhard Lamel}

\address{Faculty of mathematics, University of Vienna, Austria, }
\email{bernhard.lamel@univie.ac.at }

\begin{abstract}
We extend the celebrated theorem of Kellogg for conformal mappings to the minimizers of Dirichlet energy. Namely we prove that a diffeomorphic minimizer of Dirichlet energy of Sobolev mappings between doubly connected domains $D$ and $\Omega$ having $\mathscr{C}^{n,\alpha}$ boundary is $\mathscr{C}^{n,\alpha}$ up to the boundary, provided $\Mod(D)\ge \Mod(\Omega)$. If $\Mod(D)< \Mod(\Omega)$ and $n=1$ we obtain that the diffeomorphic minimizer has $\mathscr{C}^{1,\alpha'}$ extension up to the boundary, for $\alpha'=\alpha/(2+\alpha)$. It is crucial that, every diffeomorphic minimizer of Dirichlet energy has a very special Hopf differential and this fact is used to prove that every  diffeomorphic minimizer of Dirichlet energy  can be locally lifted to a certain minimal surface near an arbitrary point inside and at the boundary. This is a complementary result of an existence results proved by T. Iwaniec, K.-T. Koh, L. Kovalev, J. Onninen (Inventiones 2011).

\end{abstract}

\maketitle
\tableofcontents

\section{Introduction and the main results}
In this paper, we consider two doubly connected domains  $D$ and $ \Omega$  in the complex plane $\C$.
The \emph{Dirichlet energy}
 of a diffeomorphism $f\colon D \to \Omega$
 is defined by
\begin{equation}\label{ener1}
\E[f]= \int_{D}\|Df\|^2 \, d \lambda  = 2 \int_{D}\left(\abs{\partial f}^2 + \abs{\bar \partial f}^2\right) \, d \lambda,
\end{equation}
where $\|Df\|$ is the Hilbert-Schmidt norm of the differential matrix of $f$ and $\lambda$ is standard
Lebesgue measure.
The primary goal of this paper is to establish
boundary regularity
of a diffeomorphism $f\colon D\onto \Omega$ of
smallest (finite) Dirichlet energy, provided such
an $f$ exists
and the boundary is smooth. If we denote
by $J(z,f)$ the Jacobian of $f$ at the point $z$,
then~\eqref{ener1} yields
\begin{equation}\label{ener2}
\E[f] = 2\int_{D} J(z,f)\, d \lambda + 4\int_{D} \abs{\bar \partial f}^2\ge 2 \abs{\Omega}
\end{equation}
where $\abs{ \Omega}$ is the measure of $ \Omega$. In this paper we will assume that diffeomorphisms as well as Sobolev homeomorphisms are orientation preserving, so that $J(z,f)>0$. A conformal mapping of $D$ onto $ \Omega$ would be an obvious minimizer of~\eqref{ener2},
because $\bar \partial f=0$, provided it exists. Thus in
the special case where $D$ and $\Omega$ are conformally
equivalent the famous Kellogg theorem yields that the minimizer is as smooth as the boundary in
the H\"older category. For an exact statement of the Kellogg theorem, we
recall that  a function $\xi:D\to \mathbf C$ is
said to be
  uniformly $\alpha-$H\"older continuous and write
$\xi \in \mathscr{C}^{\alpha}(D)$ if
$$\sup_{z\neq w, z,w\in D}\frac{|\xi(z)-\xi(w)|}{|z-w|^\alpha}<\infty.$$
In similar way one
defines  the class  $ \mathscr{C}^{n,\alpha}(D)$
to consist of all functions
$\xi \in \mathscr{C}^n (D)$ which have
their $n$th derivative
$\xi^{(n)} \in \mathscr{C}^\alpha (D)$. A rectifiable Jordan curve $\gamma$ of the length $l=|\gamma|$ is said to be of class $\mathscr{C}^{n,\alpha}$ if its arc-length parameterization $g:[0,l]\to \gamma$ is in $\mathscr{C}^{n,\alpha}$, $n\ge 1$.
The theorem of Kellogg (with an extension
due to Warschawski, see
\cite{G, sw3, w1,w2, chp}) now states that if  $D$ and $\Omega$ are Jordan domains having
$\mathscr{C}^{n,\alpha}$  boundaries and $\omega$ is a conformal
mapping of $D$ onto $\Omega$, then $\omega \in \mathscr{C}^{n,\alpha}$.

The theorem of
Kellogg and of Warshawski has been extended in various directions, see for example the work on conformal minimal parameterization of  minimal surfaces by Nitsche \cite{nit} (see also the paper by Kinderlehrer \cite{Kid} and by F. D. Lesley \cite{Les0}), and to quasiconformal
harmonic
mappings with respect to the hyperbolic metric by Tam and Wan \cite[Theorem 5.5.]{tam}. For some other extensions and quantitative Lipschitz constants we refer to the paper \cite{Lw}.

We have the following extension of the Kellogg's theorem, which is the main result of the paper.

\begin{theorem}\label{maine0}
Let $\alpha\in (0,1)$. Assume that $D$ and $\Omega$ are two doubly connected domains in  the complex plane with $\mathscr{C}^{1,\alpha}$ boundaries. Assume that $f$ is a diffeomorphic minimizer of energy \eqref{ener1} throughout the class of all diffeomorphisms between $D$ and $\Omega$. Then $f$ has a $\mathscr{C}^{1,\alpha'}$ extension up to the boundary, with $\alpha'=\alpha$ if $\Mod(D)\ge \Mod(\Omega)$ and $\alpha'=\frac{\alpha}{2+\alpha}$ if $\Mod(D)< \Mod(\Omega)$.
\end{theorem}
For higher-degree regularity we will prove the following result:
\begin{theorem}\label{mainen}
Let $\alpha\in (0,1)$. Assume that $D$ and $\Omega$ are two doubly connected domains in  the complex plane with $\mathscr{C}^{n,\alpha}$ boundaries so that $\Mod(D)\ge \Mod(\Omega)$. Assume that $f$ is a diffeomorphic minimizer of energy \eqref{ener1} throughout the class of all diffeomorphisms between $D$ and $\Omega$. Then $f$ has a $\mathscr{C}^{n,\alpha}$ extension up to the boundary.
\end{theorem}

We will formulate some corollaries of Theorem~\ref{maine0} in Section~2, where we will describe the key point of the proof. In Section~\ref{holderi} together with the appendix below we prove that diffeomorphic minimizers  are H\"older continuous at the boundary components.  This is needed to prove the global Lipschitz continuity of such diffeomorphisms, which is done in Subsection~\ref{lipsub}. The proof of the smoothness issue is given in Section~\ref{sectioq}.
Section~\ref{vised} contains the proof of the main results. The last section is devoted to an open problem.

 The following  existence result was proved in \cite{iwa}:

\begin{proposition}\label{q4}
Suppose that $D$ and $ \Omega$ are bounded doubly connected domains in $\C$ such that $\Mod D\le \Mod \Omega$.
Then there exists a diffeomorphism $h$ of finite Dirichlet energy, which minimizes the energy amongst
all diffeomorphisms; that is, $\E[h]=\inf\{\E[f]: f\text{ is a diffeomorphism between $D$ and $\Omega$}\}$.
Moreover, $h$ is harmonic and it is unique up to a conformal automorphism of $D$.
\end{proposition}

The most important issue in proving Proposition~\ref{q4} was to establish some key properties of Noether harmonic
which we gather in the next subsections.
\subsection{Noether harmonic mappings}
We recall that a mapping  $g:D\to \Omega$ is said to be {\em Noether harmonic} (see  \cite{helein2})  if
\begin{equation}\label{stat}
\frac{\mathrm{d}}{\mathrm{dt}}\bigg|_{t=0}{\mathcal{E}}[g\circ \phi_t^{-1}]=0
\end{equation}
for every family of diffeomorphisms $t\to \phi_t\colon
\Omega\to\Omega$ depending smoothly on the real parameter $t$ and satisfying
$\phi_0=\id$. To be more exact, that means that the mapping $\Omega\times [0,\epsilon_0]\ni (t,z)\to \phi_t(z)\in \Omega $ is a smooth mapping for some $\epsilon_0>0$. Not every  Noether harmonic mapping $h$ is a harmonic mapping, however if the mapping $g$ is a diffeomorphism, then it is harmonic, i.e. it satisfies the equation $\Delta g=0$.

\subsection{Some key properties of Noether harmonic diffeomorphisms}\label{sube}
 The following properties of Noether harmonic mappings are derived in the proof of \cite[Lemma~1.2.5]{jost}.
 Assume that $g\colon D\to \Omega$ is Noether harmonic. Then
\begin{itemize}
\item[1.] The Hopf differential of $g$, defined by $\varphi:=  g_z\overline{g_{\bar z}}$, which a priori belongs to $L^1(D)$, is holomorphic.
\item[2.] If $\partial D$ is $\mathscr{C}^{1,\alpha}$-smooth then $\varphi$ extends continuously to $\overline{D}$, and the quadratic differential $\varphi \, dz^2$  is real on each boundary curve of $D$.
\end{itemize}

Using those key properties, \cite[Lemma 6.1]{iwa} (and \cite{kal0}) show the following:
If  $D=\A(r,R)$, where $0<r<R<\infty$, is a circular annulus centered at origin,
$\Omega$ a doubly connected domain, and $g \colon D \to \Omega$ is a stationary deformation, then there
exists a real constant $c\in\mathbf{R}$ such that
\begin{equation}\label{hopf1}
g_z\overline{g_{\bar z}} \equiv \frac{c}{z^2}\qquad \text{in }D
\end{equation}
We recall that every doubly connected domain $D\subset \C^2$ whose inner boundary is not
just one point is conformally equivalent to such
an annulus $D=\A(r,R)$. The conformal invariant $\Mod D := R/r$ is called the conformal
modulus of $D$.

The constant  $c$ appearing in \eqref{hopf1} is related to the
conformal modulus by the following proposition.
\begin{proposition}\label{cpositive}\cite[Corollary~5.2]{kal0}. If $g\colon D\to \Omega$
is a  Noether harmonic deformation, then we have
\[  \begin{cases}
c>0 & \text{ if } \Mod D < \Mod \Omega, \\
c=0 & \text{ if } \Mod D = \Mod \Omega, \\
c<0 & \text{ if } \Mod D > \Mod \Omega.
\end{cases}\]
\end{proposition}

We next recall that  sense preserving mapping $w$ of class ACL between two planar domains $\X$ and $\Y$ is called $(K, K')$-quasi-conformal if
\begin{equation}\label{map}\|Dw\|^2\le 2KJ(z,w)+K',\end{equation}
for almost every $z\in \X$. Here $K\ge 1, K'\ge 0$, $J(z,w)$ is the Jacobian of $w$ in $z$ and $\|Dw\|^2=|w_x|^2+|w_y^2|=2|w_z|^2+2|w_{\bar z}|^2$. For a related definition for mappings between surfaces the reader is referred to \cite{simon}.

Noether-harmonic maps, and in particular minimizers,  belong to the class of $(K,K')$ quasiconformal mappings, for
a $(K,K')$ which is nicely related to the data $c$ and $\Mod D$:

\begin{lemma}\label{popi}\cite{kal}
Every sense-preserving Noether harmonic map $g:\A(\rho,1)\to \Omega$ is $(K,K')$ quasiconformal, where $$K=1 \ \ \text{and}\ \ K'=\frac{2\abs{c}}{\rho^2},$$ and $c$ is the constant from \eqref{hopf1}. The result is sharp and for $c=0$ the Noether harmonic map is $(1,0)$ quasiconformal, i.e. it is a conformal mapping. In this case $\Omega$ is conformally equivalent to $\A(\rho,1)$.
\end{lemma}

Assume that $g:[0,\ell]\to \Gamma$ is the arc-length parameterization of a rectifiable Jordan curve $\Gamma$. Here $\ell=|\Gamma|$ is the length of $\Gamma$. We say that a continuous mapping $f: \mathbf{T}\to \Gamma$ of the unit circle onto $\Gamma$ is monotone if there exists a monotone function $\phi:[0,2\pi]\to [0,\ell]$ such that $f(e^{it})=g(\phi(s))$. In a similar way we define a monotone function between $\rho \mathbf{T}:= \{z: |z|=\rho\}$ and $\Gamma$.
In view of \cite[Proposition~5]{koh2} and Proposition~\ref{cara} below we can formulate the following simple lemma.
\begin{lemma}\label{newcara}
Assume that $f$ is a diffeomorphic minimizer of Dirichlet energy between the annuli $\A(\rho,1)$
and the doubly connected domain $\Omega$, which is  bounded by the outer boundary $\Gamma$ and inner boundary $\Gamma_1$.
Then $f$ has a continuous extension  to the boundary and the boundary mapping is monotone in both boundary curves.
\end{lemma}

\section{Some corollaries and the strategy of the proof}

Theorem~\ref{maine0} and Proposition~\ref{q4} imply the following result:
\begin{corollary}
Assume that $D$ and $\Omega$ are two doubly connected domains in $\mathbf{C}$ with $\mathscr{C}^{1,\alpha}$ boundary. Assume also that $\Mod(D)\le \Mod(\Omega)$. Then there exists a minimizer $h$ of Dirichlet energy $\E$ and it has a $\mathscr{C}^{1,\alpha/(2+\alpha)}$ extension up to the boundary. Moreover it is unique up to the conformal change of $D$.
\end{corollary}

In view of Subsection~\ref{subsec}, we can state the following corollary.

\begin{corollary}
Assume  $f=u+i v:\A_\rho\to \Omega$ is a diffeomorphic minimizer of Dirichlet energy among the diffeomorphisms, where $\Omega$ is a doubly connected domain with a $\mathscr{C}^{1,\alpha}$ boundary. Then, the mapping
  $$F(z) = \left\{
                                                                     \begin{array}{ll}
                                                                       \left(u,v, 2\sqrt{c}\mathrm{Arg}\,z\right), & \hbox{for $c>0$;} \\
                                                                       \left(u,v, 2\sqrt{-c}\log \frac{1}{|z|}\,\right), & \hbox{for $c\le 0$,}
                                                                     \end{array}
                                                                   \right.
$$ is a conformal parametrisation of a minimal surface $\Sigma$, whose boundary is in $\mathscr{C}^{1,\alpha'}$, where $\alpha'=\alpha$, if $c\le 0$ and $\alpha'=\alpha/(2+\alpha)$ otherwise. If $c\le 0$, then  the surface $\Sigma$ is a doubly connected catenoidal minimal surface, whose conformal modulus is equal to $\Mod \A_\rho\ge \Mod\Omega$. If $c>0$, then the minimal surface $\Sigma$ is a helicoidal minimal surface.
\end{corollary}

The minimizer of Dirichlet energy is not always a diffeomorphism when $\Mod(D)\ge \Mod(\Omega)$. Moreover it fails to be smooth in the domain if the boundary is not smooth \cite{cris}. For more general setting we refer to \cite{duke}.
\begin{remark}
By using Lemma~\ref{popi}, the first author in \cite{kal} proved that, a minimizer of $\varrho-$energy between doubly connected domains having $C^2$ boundary is Lipschitz continuous. The  $\varrho-$energy, is a certain generalization of Euclidean energy, and we will omit details in this paper. \end{remark}

\subsection{Minimizing mappings and minimal surfaces}\label{subsec} Since $D$ is conformally equivalent to $\A_\rho=\{z: \rho<|z|<1\}$, for some $\rho\in(0,1)$,  we can assume that $D=\A_\rho$. Namely, by a Kellogg's type result of Jost (\cite{jost3}), a conformal biholomorphism of a domain $D$ with $\mathscr{C}^{n,\alpha}$ boundary onto $\A_\rho$ is $\mathscr{C}^{n,\alpha}$ continuous up to the boundary together with its inverse.
For every $p\in \partial\A_\rho$, there is a Jordan domain $A_p\subset \A_\rho$, containing a Jordan arc $T_p$ in $\partial\A_\rho$, whose interior contains $p$. Moreover in view of Lemma~\ref{newcara}, enlarging $T_p$ if necessary, we can assume that $\Gamma_p:=f(T_p)$ is a Jordan arc containing $q=f(p)$ in its interior in $\partial D$. Moreover we can assume that $A_p$ has a $C^\infty$ boundary. Assume now that $\Phi_p$ is a conformal mapping of the unit disk $\mathbf{D}$ onto $A_p$ so that $\Phi_p(p/|p|)=p$. Moreover, if $p'\neq p$, but $|p|=|p'|$ we can chose domains $A_{p'}$ to be just rotation of $A_{p}$. So all those domains $A_p$ are isometric to $A_1$ or $A_\rho$. Moreover we also can assume that $\Phi_{p'}=e^{i\varsigma} \Phi_{p}$. Then $f_p = f\circ \Phi_p$ has the representation
\begin{equation}\label{repe}
f_p(z) = g(z) +\overline{h(z)},\end{equation} where $g(z) =g_p(z) $ and $h(z)=h_p(z)$ are holomorphic mappings defined on the unit disk. Moreover $f_p$ is a sense preserving diffeomorphism and this means that $$J(z,f_p)=|g'(z)|^2-|h'(z)|^2>0.$$


From \eqref{hopf1} we have
\begin{equation}\label{keykey} f_{z}\overline{f_{\bar z}}=\frac{c}{z^2}, \ \ \ z\in \A_\rho.
\end{equation}
It follows from \eqref{keykey} and \eqref{repe} that \begin{equation}\label{kii}h_p'g_p'=c\frac{(\Phi_p'(z))^2}{\Phi_p^2(z)}.\end{equation} Then it defines locally the minimal surface by its conformal minimal coordinates, $\varphi_p=\left(\varphi_1,\varphi_2, \varphi_3\right)$, and this is crucial for our approach:
\begin{eqnarray}\label{eee1}
  \varphi_1(z)  &=& \Re (g+h)\\
 \label{eee2} \varphi_2(z) &=& \Im (g-h)  \\
  \label{eee3}\varphi_3(z) &=&  \Re (2i \sqrt{c}\log \Phi_p(z)).
\end{eqnarray}
This can be written
\begin{eqnarray}
  \varphi_1(z)  &=&\varphi_1(z_0)+ \Re \int_{z_0}^z ( g'(z) + h'(z)) dz \\
  \varphi_2(z) &=& \varphi_2(z_0)+\Re \int_{z_0}^z i(h'(z) -g'(z)) dz \\
  \varphi_3(z) &=&  \varphi_3(z_0)+\Re  \int_{z_0}^z 2i\sqrt{h'(z) g'(z)} dz.
\end{eqnarray}
Thus the Weierstrass--Enneper parameters are
$$p(z)=g'(z), q(z) = \sqrt{\frac{h'(z)}{g'(z)}}.$$
The first fundamental form is given by $ds^2 =\lambda(z)|dz|^2$, where $$\lambda(z) = \frac{1}{2}\sum_{j=1}^3 |k_j|^2.$$
Here $$k_1(z) = g'(z) + h'(z), \ \  k_2(z) = i( h'(z) - g'(z)),  \ \  k_3(z) =2i\sqrt{h'(z) g'(z)}.$$
Then as in \cite[Chapter~10]{dure}, we get
$$\lambda(z)=|p|^2(1+|q|^2)^2= |g'(z)|^2 \left(1+\frac{|g'(z)|}{|h'(z)|}\right)^2=(|g'(z)|+|h'(z)|)^2.$$

Let us note the following important fact, the boundary curve of the minimal surface defined in \eqref{eee1}, \eqref{eee2} and \eqref{eee3} is $$\varphi_p(e^{is}) = (\varphi_1(e^{is}), \varphi_2(e^{is}), \varphi_3(e^{is})),s\in [0,2\pi),$$ $p\in\partial\A_\rho$.
Its trace is not smooth in general. However the trace of curve $$z_p(e^{is})=(\varphi_1(e^{is}), \varphi_2(e^{is}))$$ is smooth as well as the function $k_3$ is smooth in a small neighborhood of $p$.  This will be crucial in proving our main results.

We will prove certain boundary behaviors of $f$ near the boundary by using the representation \eqref{repe}, and this is why we do not need global representation. The idea is to prove that $f$ is Lipschitz and has smooth extension up to the boundary locally. And this will imply the same behaviour on the whole boundary. The conformal mapping $\Phi_p$ is a diffeomorphism and it is $\mathscr{C}^{\infty}(\overline{\mathbf{D}})$, provided the boundary of $A_p$ belongs to the same class. So we will go back to the original mapping easily.

In the previous part we have showed that every minimizing mapping can be lifted locally to a certain minimal surface. In the following part we show that in certain circumstances the lifting is global.

Every harmonic mapping $f$ defined on the annulus $\A_\rho$ can be expressed (see e.g. \cite[Theorem~9.1.7]{abr}) as \begin{equation}\label{logar}f(z) = a_0 \log |z|+b_0 + \sum_{k\ne 0} (a_k z^k + \overline{b_k} \bar z^k).\end{equation} Assume now that $f$ is a diffeomorphic minimizer between $\A_\rho$ and $\Omega$ and that $c<0$, i.e. $\Mod(\A_\rho)>\Mod(\Omega)$ (see Proposition~\ref{cpositive}). Then we get the following conformal parameterization of a catenoidal minimal surface $\Sigma$, $\varphi:\A_\rho\to \Sigma$, defined by
\begin{equation}\label{lift0} \varphi(z) =\left(\Re f(z) ,\Im f(z), 2 \sqrt{-c}\log\frac{1}{|z|}\right).\end{equation}

If $a_0 = 0$, then we have the following decomposition $f(z) = b_0+g_\circ(z)+\overline{h_\circ(z)}$, where $$g_\circ(z) = \sum_{k\ne 0} a_k z^k,$$ and $$h_\circ(z)=\sum_{k\ne 0} b_k  z^k.$$ Then we get the following conformal parameterization of a minimal surface $\Sigma$, $\varphi:\A_\rho\to \Sigma$, defined by
\begin{equation}\label{lift} \varphi(z) =\left(\Re (g_\circ(z)+h_\circ(z)),\Im (g_\circ(z)-h_\circ(z)), 2 \sqrt{-c}\log\frac{1}{|z|}\right).\end{equation}
The following corollary is a consequence of Theorem~\ref{maine0} and \eqref{lift0} .
\begin{corollary}
Assume that $f:\A_\rho \to \Omega$ is a diffeomorphic minimizer of Dirichlet energy, with $\Mod \A_\rho\ge \Mod \Omega$ and $\partial \Omega\in \mathscr{C}^{1,\alpha}$. Then $f$ can be lifted to a smooth doubly connected minimal surface $\Sigma$ with $\mathscr{C}^{1,\alpha}$ boundary, and the lifting is conformal and harmonic.
\end{corollary}

Let us continue this subsection with the following explicit example.
Let \begin{equation}\label{nits}f(z)= \frac{r (R-r)}{\left(1-r^2\right) \bar z}+\frac{(1-r R) z}{1-r^2}.\end{equation} Then $f(z)$ is a harmonic mapping of the annulus $\A_r$ onto $\A_R$ that minimizes the Dirichlet energy (\cite{AIM}). Further, under notation of this subsection we have $$p(z) = \frac{1-rR}{1-r^2}$$ and $$q(z) =\frac{\sqrt{r (r-R) (1-r R)}}{\left(1-r^2\right) z}.$$ Put $\varphi_1=\Re f(z)$, $\varphi_2(z) = \Im f(z)$ and assume that $\Mod(\A_r)>\Mod(\A_R)$, i.e. $R> r$. Then we have from \eqref{lift} that
\[\begin{split}\varphi_3(z) &= \Re\int   2iq(z) dz=\Re \int 2i^2 \frac{\sqrt{r (R-r) (1-r R)}}{\left(1-r^2\right) z}dz\\&=2\frac{\sqrt{r (R-r) (1-r R)}}{\left(1-r^2\right) }\log \frac{1}{|z|}. \end{split}\] Here $\int Q(z) dz$ stands for the primitive function of $Q(z)$.
It follows that \eqref{nits} defines a global minimal surface by its conformal minimal coordinates $\varphi(z)=(\varphi_1(z), \varphi_2(z), \varphi_3(z))$. This minimal graph is a part of the lower slab of catenoid. (see Figure~1).

\begin{figure}[htp]\label{f1}
\centering
\includegraphics{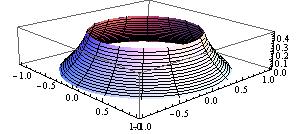}
\caption{A part of catenoid over an annulus. Here $R=2/3$ and $r=1/2$.}
\end{figure}

\begin{remark}
It follows from \eqref{nits} that, $$w(z)=\sqrt{-c}\log \frac{1}{\abs{\sqrt{f^{-1}(z)}}}$$ defines  the nonparametric minimal surface $\Sigma$. This means that $\Sigma=\{(x,y, w(z)): z\in \Omega \}$.  Moreover $\Mod(\Sigma) = \log \frac{1}{\rho}\ge \Mod \Omega.$
\end{remark}

For $c>0$, i.e. for $\Mod(\A_\rho)<\Mod(\Omega)$ we get the following counterpart.
Then we get the following conformal parametrization of a helicoidal minimal surface $\Sigma$, $\varphi:\A_\rho\to \Sigma$, defined by
\begin{equation}\label{lift01} \varphi(z) =\left(\Re f(z) ,\Im f(z), 2 \sqrt{c}\mathrm{Arg}\, z\right).\end{equation}
If $f$ has not the logarithmic  part, then we get the parametrization of a  minimal surface \begin{equation}\label{lift3} \varphi(z) =\left(\Re (g_\circ+h_\circ),\Im (g_\circ-h_\circ), 2 \sqrt{c}\mathrm{Arg} z\right).\end{equation}
In particular, if $\A_\rho =A_r$ and $\Omega=\A_R$, so that $R<r$, then $$\varphi_3(z)=-2\frac{\sqrt{r (r-R) (1-r R)}}{\left(1-r^2\right) }\mathrm{Arg}(z). $$ In particular, if $r=2/3$ and $R=1/3$ then this minimal surface over the annulus $A_r$ is shown in Figure~2.
\begin{figure}[htp]\label{f12}
\centering
\includegraphics{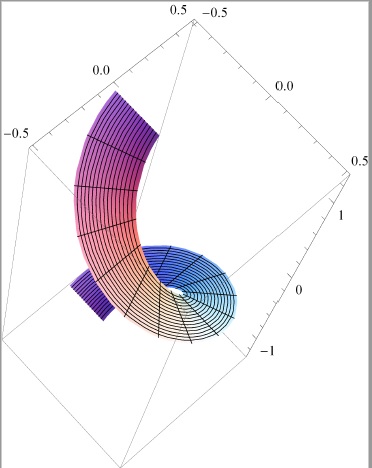}
\caption{A part of helicoid over an annulus. Here $R=1/2$ and $r=2/3$.}
\end{figure}

We finish  this section with a lemma needed in the sequel

\begin{lemma}\label{male} Let $p\in\mathbf{T}=\partial{\mathbf{D}}$.
a) Assume that $\Phi$ is a holomorphic  mapping of the unit disk into itself so that $\Phi(p) = p$ and $\Phi$ has the derivative at $p$. Then $$\Phi'(p)\ge \frac{1-|\Phi(0)|}{1+|\Phi(0)|}>0.$$

b) Assume that $\Phi$ is a holomorphic mapping of the unit disk into the exterior of the disk $r \mathbf{D}$ with $\Phi(p)=rp$. Then $$\Phi'(p)<r\frac{r-|\Phi(0)|}{|\Phi(0)|+r}<0.$$
\end{lemma}
To prove Lemma~\ref{male} we recall the boundary Schwarz lemma (\cite{garnet}) which states the following.
\begin{lemma}[Boundary Schwarz lemma] Let $f :\mathbf{D}\to \mathbf{D}$ be a holomorphic function. If $f$ is holomorphic at $z =1$ with $f (0)
=0$ and $f (1)
=1$,
then $f'(1)\ge 1$. Moreover, the inequality is sharp.
\end{lemma}
\begin{proof}[Proof of Lemma~\ref{male}]
Assume that $p=1$. Otherwise consider the function $\Phi_1(z) = \frac{1}{p}\Phi(zp)$.
Consider $$F(z) = \frac{(1-\overline{\Phi(0)}) (\Phi(z)-\Phi(0))}{(1-\Phi(0)) (1-\overline{\Phi(0)}\Phi(z))}.$$
Then $$F'(1) = \frac{1+|\Phi(0)|}{1-|\Phi(0)|}\Phi'(1).$$ Since $F(0)=0$, $F(1) = 1$, it follows that $F$ satisfies the boundary Schwarz lemma, and therefore $F'(1)$ is a real positive number bigger or equal to $1$. This implies a).

In order to prove b), consider the auxiliary function $g(z) =\frac{r}{\Phi(z)}$. By applying a) to $g$ we get $$g'(1)\ge \frac{1-|g(0)|}{1+|g(0)|}.$$ Since
$$\Phi'(z) =\frac{-rf'(z)}{\Phi^2(z)},$$ we get $$\frac{-r \Phi'(1)}{r^2}\ge  \frac{1-|g(0)|}{1+|g(0)|}$$ and so
$$ -\Phi'(r)\ge r  \frac{1-\frac{r}{|\Phi(0)|}}{1+\frac{r}{|\Phi(0)|}}=r\frac{|\Phi(0)|-r}{|\Phi(0)|+r}.$$ This finishes the proof.
\end{proof}

\section{H\"older property of minimizers}\label{holderi}

 In this section we prove  that the minimizers of the energy are global H\"older continuous  provided that the boundary is  $\mathscr{C}^{1}$.

We first formulate the following result
\begin{proposition}[Caratheodory's theorem  for $(K,K')$ mappings]\cite{kalmat}\label{cara} Let $W$ be
a simply connected domain in $\overline{\mathbb{C}}$ whose boundary
has at least two boundary points such that $\infty\notin \partial
W$. Let $f : \mathbf{D} \rightarrow W$ be a  continuous mapping of
the unit disk $\mathbf{D}$ onto $W$ and $(K,K')$ quasiconformal near
the boundary $\mathbf T$.

Then $f$ has a continuous extension up to
the boundary if and only if  $\partial W$ is locally connected.\\
\end{proposition}

 Let $\Gamma\in \mathscr{C}^{1,\mu}$, $0<\mu\le 1$, be a
Jordan curve and let $g$ be the arc length parameterization of
$\Gamma$ and let $l=|\Gamma|$ be the length of $\Gamma$.  Let
$d_\Gamma$ be the distance between $g(s)$ and $g(t)$ along the curve
$\Gamma$, i.e.
\begin{equation}\label{kernelar3}d_\Gamma(g(s),g(t))=\min\{|s-t|,
(l-|s-t|)\}.\end{equation}

A closed rectifiable Jordan curve $\Gamma$ enjoys a $b-$ chord-arc
condition for some constant $b> 1$ if for all $z_1,z_2\in \Gamma$
there holds the inequality
\begin{equation}\label{24march}
d_\Gamma(z_1,z_2)\le b|z_1-z_2|.
\end{equation}
It is clear that if $\Gamma\in \mathscr{C}^{1}$ then $\Gamma$ enjoys a
chord-arc condition for some $b=b_\Gamma>1$. In the following lemma we use the notation $\Omega(\Gamma)$ for a Jordan domain bounded by the Jordan curve $\Gamma$. Similarly, $\Omega(\Gamma, \Gamma_1)$ denotes the doubly connected domain between two Jordan curves $\Gamma$ and $\Gamma_1$, such that $\Gamma_1\subset \Omega(\Gamma)$.

The following lemma is a $(K,K')$-quasiconformal version of
\cite[Lemma~1]{SW}.
\begin{lemma}\label{newle}
Assume that the Jordan curves $\Gamma,\Gamma_1$ are in the class $\mathscr{C}^{1,\alpha}$.
Then there is a constant $B>1$, so that  $\Gamma$ and $\Gamma_1$ satisfy $B-$ chord-arc condition and for every $(K,K')-$ q.c.  mapping $f$ between the
annulus $\A_\rho$ and the doubly connected domain $\Omega=\Omega(\Gamma,\Gamma_1)$ there exists a positive constant $L=L(K,K',B, \rho, f)$ so that  there holds \begin{equation}\label{enjte}|f(z_1)-f(z_2)|\le
L|z_1-z_2|^\beta\end{equation} for $z_1,z_2\in \mathbf T$ and $z_1,z_2\in \rho\mathbf{T}$ for $\beta
= \frac{1}{K(1+2B)^2}.$
\end{lemma}
See appendix below for the proof of Lemma~\ref{newle}.
We now can state the following proposition:
\begin{proposition}\label{propop}
Let $f$ be a diffeomorphic minimizer of the Dirichlet energy between the annulus $\A_\rho$ and the doubly connected domain $\Omega(\Gamma, \Gamma_1)$, where $\Gamma$ and $\Gamma_1$ are $\mathscr{C}^{1,\alpha}$ Jordan curves. Then $f$ is H\"older continuous on $\overline{\A_\rho}$.
\end{proposition}
The proof of Proposition~\ref{propop} follows from Lemma~\ref{newle}, Lemma~\ref{heliu} below and compactness property of $\overline{\A_\rho}$.

\section{Proof of Theorem~\ref{maine0}}

By repeating the proofs of corresponding result in \cite{nit} we can formulate the following result.

\begin{proposition}\label{equation}
Assume that $\Gamma$ is a Jordan curve in $\mathbf{R}^3$ and assume that $\vec{X}(z)=(X_1,X_2,X_3):\mathbf{D}\to \mathbf{R}^3$ is a minimal graph so that $\vec{X}(\mathbf{T})=\Gamma$. Assume that $\vec{X}$ is H\"older continuous in an arc $T_p\subset \mathbf{T}$ containing $p$ in its interior. If the arc $T_p$ of $\mathbf{T}$ is mapped onto the arc $\Gamma_p\subset \Gamma$ so that $\Gamma_p\in \mathscr{C}^{1,\alpha}$, $0<\alpha<1$, then $\vec{X}$ is $\mathscr{C}^{1,\alpha}$ in a small neighborhood of $p=e^{it_0}$ i.e. in a domain $D_{p,\delta}=\{z=re^{it}: 1/2\le r<1, t\in(-\delta+t_0, \delta+t_0)\}$.
\end{proposition}

From time to time in the proof we will use the notation $D_p$ or $D_\delta$ instead of $D_{p,\delta}$, but the meaning will be clear from the context.

   The proof of Proposition~\ref{equation} depends deeply on the proof of a similar statement in \cite{nit}. We observe that, almost all results proved in \cite{nit} are of local nature (see \cite[Lemma~5,~Lemma~6,~Lemma~7]{nit}), thus we will not write the details here.

We want to mention that also Lesley in \cite[p.~125]{Les0} have made a similar remark. Further a similar explicit formulation to related to Proposition~\ref{equation}  has been stated as Theorem~1 in Section 2.3. of the book of Dierkes, Hildebrandt and Tromba \cite{dht}.

   Since the minimising property is preserved under composing by a conformal mapping, in view of the  original Kellogg's theorem \cite{G}, we can assume that the domain is $\A_\rho=\{z: \rho<|z|<1\}$.

   On the other hand, the minimising harmonic mapping  has  the local representation \eqref{eee1}. Here $\Phi_p$ is a $ C^\infty$ diffeomorphism, and it does not cause any difficulty.

Let $p\in\partial\A_\rho$ be arbitrary, say $|p|=1$ (the other possibility is $|p|=\rho$). Because the boundary mapping is continuous and monotone, in view of Lemma~\ref{newcara}, it follows that, there is a neighborhood $T_p$ which is mapped onto the arc $\Gamma_p\subset \partial \Omega$. Therefore by Theorem~\ref{equation}, having in mind the notation from subsection~\ref{subsec}, the mapping $$\vec{X}(z)=\vec{X}_p(z)=\{\Re f_p(z), \Im f_p(z), \Re (2i \sqrt{c}\log \Phi_p(z))\}$$ is $\mathscr{C}^{1,\alpha}$ in a neighborhood of $p$, provided the boundary arc is of the same class. But we do not know that $\vec{X}(T_p)\in \mathscr{C}^{1,\alpha}$. We only know that  $\Phi_p$ is a priori in $C^\infty(\overline{\mathbf{D}})$ and $\Gamma_p=f_p(T_p)\in \mathscr{C}^{1,\alpha}$. This will be enough for the proof.

\subsection{Proof of Lipschitz continuity}\label{lipsub}
We will prove the following lemma needed in the sequel.

\begin{lemma}\label{lemaimpo}
Assume that $f=u+iv:\A_\rho\to \Omega$ is a diffeomorphic minimizer, where $\A_\rho=\{z: \rho<|z|<1\}$ and assume that $\partial\Omega\in \mathscr{C}^{1,\alpha}$. Then $f$ is Lipschitz continuous.
\end{lemma}
\begin{proof}
We use the notation from Subsection~\ref{subsec}.  The constant $C$ that appear in the proof is not the same and its value can vary from one to the another appearance.
Assume also  $q\in \Gamma=\partial \Omega$, and, by using a rotation and a translation (if it is necessary)  we can assume that  $q=0$, and the tangent line of $\Gamma$ at $q$ is the real axis. Post-composing by a such Euclidean isometry, the Euclidean harmonicity is preserved.
Then in a small neighborhood of $q$,  $\Gamma$ has the following parameterization $\gamma(x)=(x, \phi(x))$, $x\in(-x_0, x_0)$, so that $\phi(0)=\phi'(0)=0$.
Assume also  that, $p=1$ and  $f(1) = q=0$. And assume that for a small angle $\Lambda=\Lambda_p=\{e^{i\theta}: |\theta|\le \epsilon\}$ we have $f(\Lambda) \subset\gamma(-x_0,x_0)$. We can assume also that $x_0$ is a  small enough positive constant global for all points $q\in\partial \Omega$.
We want to localize the problem. We only need to prove that $f$ is $\mathscr{C}^{1,\alpha'}$ in a small neighborhood of $1$. We also work with $f_p=f\circ \Phi_p:\mathbf{D}\to f(A_p)$ instead of $f$, where $\Phi_p(p/|p|)=p$, and assume that $\gamma(-x_0,x_0)\subset \partial A_p$ for every $p\in\partial\A_\rho$. We will from time to time use notation $f$ instead of $f_p$, since they behave in the same way in a small neighborhood of $p$, because $\Phi_p$ is a priori in $\mathscr{C}^\infty$

Thus, there exists a function $x:\Lambda\to \mathbf{R}$ so that $$f(e^{it})=(u(e^{it}), v(e^{it}))=\gamma(x(e^{it}))=(x(e^{it}), \phi(x(e^{it}))).$$
We will also from time to time use notation $x(t)$ instead of $x(e^{it})$. Similarly $v(t)$ instead of $v(e^{it})$.

Now we have $v=\Im(f)=\Im (g+\overline{h})=\Im (g-h)=\Re (i (h-g))$ and therefore, \begin{equation}\label{begeq0}v_\theta =\Re (z (g'-h')).\end{equation}
Because $\Gamma_p\in \mathscr{C}^{1,\alpha}$ we have as in \cite[eq.~3]{nit}, the following relation
\begin{equation}\label{marmar}|\phi(s)-\phi(t)|\le C|s-t|\{\min\{|s|^\alpha,|t|^\alpha\}+|t-s|^\alpha\},\ \ |t|<t_0, |s|<t_0.\end{equation}
The constant $C$ and  $t_0$ are the same for all points $p\in \partial\A_\rho$.
 Recall that $p=1$ and $f(1)=0$. By using translations and rotations in the domain and image domain, we will obtain this property, and therefore we do not loos the generality.

 Further  $$\frac{|\phi(x) - \phi(0)-\phi'(0)x|}{|x|^{1+\alpha}}= \frac{|\phi'(\theta x)-\phi'(0)|}{|x|^\alpha}\le C,$$ where $\theta \in(0,1)$.

Since

\begin{equation}\label{vtheta}v(e^{it})=\phi(x(e^{i t}))\end{equation}
 we get \begin{equation}\label{maybe}|v(e^{it})-v(1)|=|\phi(x(e^{it}))|\le C|x(e^{it})|^{1+\alpha}.\end{equation}
 Now, the following sequence of the  inequalities follow from \eqref{marmar}, \eqref{maybe} and Lemma~\ref{newle}.
\begin{equation}\label{impo}|v(e^{it})-v(1)|\le C |t|^{\beta(1+\alpha)},\end{equation}
and
\begin{equation}\label{alexi}\begin{split}|v(e^{it})&-v(e^{is})|=|\phi(x(t))-\phi(x(s))|\\&\le C|x(s)-x(t)|\{\min\{|x(s)|^\alpha,|x(t)|^\alpha\}+|x(t)-x(s)|^\alpha\}\end{split}\end{equation} and so
\begin{equation}\label{mimi}\begin{split}|v(e^{it})-v(e^{is})|\le C L_0^{1+\alpha}|s-t|^\beta\{\min\{|s|^{\alpha\beta},|t|^{\beta \alpha}\}+|t-s|^{\beta\alpha}\}.\end{split}\end{equation}
Here \begin{equation}\label{loo}L_0 =L\Phi_0, \quad \text{ where } \Phi_0 = \sup_{|z|=1, p\in \partial \A_\rho}|\Phi_p'(z)|, \end{equation} where $L$ is defined in Lemma~\ref{newle}.

In order to continue  we collect some results from \cite{nit} and \cite{G}.

First we formulate \cite[Lemma~7]{nit} and a relation from its proof:
\begin{lemma}\label{loclema}
Assume that $F$  is a bounded holomorphic mapping defined in the unit disk, so that $| F|\le M$ in $\mathbf{D}$. Further assume that for a constants  $0\le \delta$, $0\le \eta,\mu\le \pi/2$ so that for almost every $-\delta\le t,s\le \delta$ we have
$$|\Re F(t)-\Re F(s)|\le M |t-s|^\mu \{\min\{|t|^\eta,|s|^\eta\}+|t-s|^\eta\}.$$ Then for $\zeta=\tau e^{is}$, with $|s|\le \delta/2$, $1/2 \le \tau\le 1$  we have the estimates

\begin{equation}\label{estimate}|F'(\zeta)|\le \left\{
                        \begin{array}{ll}
                          M_1|s|^\eta(1-\tau)^{\mu -1}+M_2(1-\tau)^{\mu +\eta-1}+M_3, & \hbox{if $\mu +\eta<1$;} \\
                          M_1|s|^\eta(1-\tau)^{\mu -1}+M_2\log \frac{1}{1-\tau}+M_3, & \hbox{if $\mu +\eta=1$;} \\
                          M_1|s|^\eta(1-\tau)^{\mu -1}+M_2, & \hbox{if $\mu<1 \wedge \mu +\eta>1$;} \\
                          M_1|s|^\eta \cdot \log \frac{1}{1-\tau} +M_3, & \hbox{if $\mu=1$;} \\
                          M_1, & \hbox{if $\mu>1$;}
                        \end{array}
                      \right.
\end{equation}
and
\begin{equation}\label{rho}
|F(\tau)-F(1)|\le \left\{
                    \begin{array}{ll}
                      N(1-\tau)^{\mu +\eta}, & \hbox{ if $\mu+\eta<1$;} \\
                        N(1-\tau)\log \frac{1}{1-\tau}, & \hbox{ if $\mu+\eta=1$;} \\
                        N(1-\tau), & \hbox{ if $\mu+\eta>1$,}
                    \end{array}
                  \right.
\end{equation}
and

\begin{equation}\label{theta}
|F(e^{is})-F(1)|\le \left\{
                    \begin{array}{ll}
                      N|s|^{\mu +\eta}, & \hbox{ if $\mu+\eta<1$;} \\
                        N|s|\log \frac{1}{|s|}, & \hbox{ if $\mu+\eta=1$;} \\
                        N|s|, & \hbox{ if $\mu+\eta>1$.}
                    \end{array}
                  \right.
\end{equation}
Here $N$, $M_1,M_2,M_3$ depends on $M, \eta, \mu$ and $\delta$.
\end{lemma}

By repeating the proof of the theorem of Hardy and Littlewood, \cite[Theorem~3, p.\ 411]{G} and \cite[Theorem~4, p.\ 414]{G}, we can state the following two theorems.

\begin{lemma}\label{hali} Let $\mu\in(0,1)$ and let $D_{\delta}=\{z=re^{i(s+s_0)}: 1/2\le r\le 1, s\in(-\delta,\delta)\}$.
Assume that $f$ is a holomorphic mapping defined in the unit disk so that $$|f'(z)|\le M(1-|z|)^{\mu - 1},$$ where $0<\mu<1$ and $z\in D_{\delta}$. Then the radial limit $$\lim_{\tau\to 1-0}f(\tau e^{i\theta})=f(e^{i\theta})$$ exists for every $\theta\in (-\delta+s_0,\delta+s_0)$ and we have there the inequality
$$|f(w)-f(w')|\le N |w-w'|^\mu, \ \ w,w'\in D_{\delta},$$ where $N$ depends on $M$ and $\mu$.  The converse is also true.
\end{lemma}

\begin{lemma}\label{heliu} Let $\mu\in(0,1)$.
Assume that $f$ is continuous harmonic  mapping on the closed unit disk and satisfies on a small arc $\Lambda=\{e^{i\theta}: |\theta-s_0|<\delta\}$ the  condition: $$|f(e^{is})-f(e^{it})|\le A|t-s|^\mu, \ \ e^{it}, e^{is}\in\Lambda,$$  for  almost every points $s$ and $t$. Then $f$ satisfies the H\" older condition $$|f(z) - f(w)|\le B |z-w|^\mu$$ for $z,w\in D_{\delta}=\{z=r e^{is}: 1/2\le r\le 1, s\in (-\delta+s_0,\delta+s_0)\}$.
\end{lemma}

We now reformulate a result  of Privalov \cite[p.~414,~Theorem~5]{G} in its local form (w.r.t. the boundary).
\begin{lemma}\label{privalov}
Let $\mu\in(0,1)$. Assume that $f=u + iv$ is a holomorphic bounded function defined on the unit disk $\mathbf{D}$ and assume that $u$ satisfies the condition    $|u(e^{it})-u(e^{is})|\le M |e^{it}-e^{is}|^\mu$, for almost every $s$ and $t$ so that  $|s-s_0|<\delta$ and $|t-s_0|<\delta$. Then there is a constant $N$ depending on $M$ and $\mu$ so that $|f(z) -f(w)|\le N|z-w|^\mu$ for $z,w\in D_{\delta}$, where $$D_\delta=\{z=re^{is}: 1/2 \le r\le 1, \ \ |s-s_0|\le \delta\}.$$
\end{lemma}

\begin{proof}[Proof of Lemma~\ref{privalov}]
From Schwarz formula we have

$$f(\zeta)=\frac{1}{2\pi}\int_{0}^{2\pi}u(e^{it})\frac{e^{it}+\zeta}{e^{it}-\zeta}dt +i C.$$ Thus

$$f'(\zeta)=\frac{2}{2\pi}\int_0^{2\pi}\frac{u(e^{it})e^{it} dt}{(e^{it}-\zeta)^2}=\frac{1}{\pi}\int_0^{2\pi}\frac{u(e^{it})-u(e^{is})}{(e^{it}-\zeta)^2}e^{it}  dt,\  \ \zeta=re^{is}. $$
Let $\zeta=re^{is}\in D_\delta$.
Then we get $$|f'(\zeta)|\le \frac{1}{2\pi}\int_{-\pi}^\pi \frac{|u(e^{i(s+t)})-u(e^{is})|}{1-2r \cos t+r^2 }d t.$$
If $t\in [-\pi,\pi]$, then $$1-2r \cos t+r^2\ge (1-r)^2+ \frac{4r}{\pi^2 }t^2.$$
Further, if $s\in (s_0-\delta, s_0+\delta)$, $t\in (-\delta,\delta)$  then we get $$|u(e^{i(s+t)})-u(e^{is})|\le K|t|^\mu.$$
If $t\in [-\pi,\pi]\setminus  (-\delta, \delta)$, then $$|u(e^{i(s+t)})-u(e^{is})|\le 2M \leq\frac{2M}{\delta^\mu}|t|^\mu .$$
The conclusion is that $$|f'(\zeta)|\le \frac{N}{(1-|\zeta|)^{1-\mu}},$$ for $\zeta\in D_\delta.$
Then from Lemma~\ref{hali} we get the desired result.

\end{proof}

Repeating the proof of the preceding lemma, we also obtain the following pointwise statement.

\begin{lemma}\label{privalov2}
Let $\mu\in(0,1)$ and $\delta>0$. Assume that $f=u + iv$ is a holomorphic bounded function defined on the unit disk $\mathbf{D}$ and assume that $u$ satisfies the condition $|u(e^{it})-u(e^{is_0})|\le M |e^{it}-e^{is_0}|^\mu$, for  almost every $t$: $|t-s_0|<\delta$. Then there is a constant $N$ so that \[|f'(r e^{is_0})|\le \frac{N}{(1-r)^{1- \mu }}\] for $0<r<1$.
\end{lemma}

Now we continue the proof of Lemma~\ref{lemaimpo}. Observe that $\beta<1/2$ and so $\beta(1+\alpha)<1$.
For $F_p(z) = i (h_p(z)-g_p(z))$ we get from \eqref{estimate}
\begin{equation}\label{aleva}|F'_p(\tau)|\ \le C (1-\tau)^{(1+\alpha)\beta-1},\end{equation} for $1/2\le \tau<1$. Since \begin{equation}\label{locglob}F_p\circ \Phi_p^{-1}(z) = F_1\circ \Phi_1^{-1}(z),\end{equation} for $z\in \Phi_1(\mathbf{D})\cap\Phi_p(\mathbf{D})=A_1\cap A_p$.  (and for $|p|=\rho$, $F_p\circ \Phi_p^{-1}(z) = F_\rho\circ \Phi_{\rho}^{-1}(z)$, for $z\in A_\rho\cap A_p)$), we get that \begin{equation}\label{ariani}|F'_p(z)|\ \le C (1-|z|)^{(1+\alpha)\beta-1},\end{equation} for all $z\in D_{p,\epsilon}:=\{z=re^{it}: 1/2\le |z|<1, t\in(-\epsilon,\epsilon)\} $. Further since $p=1$ is not a special point we get that \eqref{ariani}, is valid for all $z\in D_{p,\epsilon}:=\{z=re^{it+it_0}: 1/2\le |z|<1, t\in(-\epsilon,\epsilon)\} $. Here $p=e^{it_0}$ or $p=\rho e^{it_0}$. Recall that $\epsilon>0$ is small enough so that, for every point $q\in \partial\Omega$, the graph of $\partial \Omega$ after a rotation and translation has the form $(x(e^{it}),\phi(x(e^{it}))), t\in(-\epsilon,\epsilon)$.

Then from \eqref{ariani} and Lemma~\ref{hali} we get that $F_p$ is $\mathscr{C}^{0,(1+\alpha)\beta}$ in $D_{p,\epsilon}$.

Let \begin{equation}\label{gep}G_p(z) = g_p(z) + h_p(z).\end{equation}
Then we also have \begin{equation}\label{locglob1}G_p\circ \Phi_p^{-1}(z) = G_1\circ \Phi_1^{-1}(z),\end{equation} for $z\in \Phi_1(\mathbf{D})\cap\Phi_p(\mathbf{D})=A_1\cap A_p$.
Then we have \begin{equation}\label{minimali}\left({G'_p}(z)\right)^2 + \left({F'_p}(z)\right)^2=4 {g'_p}(z) {h'_p}(z)=4
\left( \frac{\Phi'_p(z)}{{\Phi_p}(z)} \right)^2 .\end{equation}

Since the right hand side of \eqref{minimali}
 is bounded, it follows  that $G_p$ is $(1+\alpha)\beta$ H\"older continuous. Namely

\[\begin{split}\left|G_p'(z) (1-|z|)^{1-(1+\alpha)\beta}\right|^2&\le \left|2\frac{\Phi'(z)}{\Phi_p(z)}(1-|z|)^{1-(1+\alpha)\beta}\right|^2\\&+\left|F_p'(z) (1-|z|)^{1-(1+\alpha)\beta}\right|^2\le N_0.\end{split}\]
 Now we have
\begin{equation}\label{reper}h_p=\frac{1}{2} (i F_p+G_p), \ \ \ g_p=\frac{1}{2} (-i F_p+G_p).\end{equation}
Since $f_p=g_p+\overline{h}_p$ and $f_p(e^{it})=\gamma(x(e^{it}))$, where $\gamma(x) = (x,\phi(x))\in \mathscr{C}^{1,\alpha}$, it follows that $x(e^{it})=\Re G(e^{it})$ and therefore
 \begin{equation}\label{xcp}x\in \mathscr{C}^{0,(1+\alpha)\beta}(\partial D_{p,\epsilon}\cap \partial \mathbf{D}).\end{equation}

Chose $\beta<1/2$ so that none of numbers $(1+\alpha)^k \beta $ is equal to $1$ for every $k$. Let $n$ be so that $(1+\alpha)^n \beta<1<(1+\alpha)^{n+1}\beta$. Then by successive application of the previous procedure we get $$|F'_p(z)|\le M(1-|z|)^{(1+\alpha)^n\beta  - 1},z=\rho e^{is},  1/2<\rho<1, \ \ s\in(-\epsilon,\epsilon), $$  and
$$|G'_p(z)|\le M(1-|z|)^{(1+\alpha)^n\beta  - 1},z=\rho e^{is},  1/2<\rho<1, \ \ s\in(-\epsilon, \epsilon). $$
Then we get $$|F_p(w)-F_p(w')|\le N |w-w'|^{(1+\alpha)^n\beta}, \ \ w,w'\in D_\epsilon,$$

$$|G_p(w)-G_p(w')|\le N |w-w'|^{(1+\alpha)^n\beta}, \ \ w,w'\in D_\epsilon,$$

where $N$ depends on $M$ and $\mu$ and so $$|F_p(e^{it})-F_p(e^{is})|\le  N |s-t|^{(1+\alpha)^n\beta}, $$
and
$$|G_p(e^{it})-G_p(e^{is})|\le  N |s-t|^{(1+\alpha)^n\beta}, $$
for $\ \ |s|<\epsilon, |s|<\epsilon$. Thus $$x\in \mathscr{C}^{0,(1+\alpha)^n\beta}(\partial D_{p,\epsilon}\cap \mathbf{D}),$$ and, as in \eqref{alexi} and \eqref{mini} we get
\[\begin{split}|f_p(e^{it})&-f_p(e^{is})|\\& \le C L_0^{1+\alpha}|s-t|^{(1+\alpha)^n\beta}\{\min\{|s|^{{(1+\alpha)^n\alpha\beta}},|t|^{{(1+\alpha)^n\alpha\beta}}\}+|t-s|^{{(1+\alpha)^n\alpha\beta}}\}.\end{split}\]
From Lemma~\ref{loclema}, for $\mu = (1+\alpha)^n\beta$ and $\eta=(1+\alpha)^n\alpha\beta$, by choosing $s=0$, we get \begin{equation}\label{fzetabar}|{F_p}'(\tau)|\le M_2,\ \ \tau\in(1/2,1).\end{equation}

Since the functions ${F'_p}(z)$ and $4\frac{(\Phi'(z))^2}{{\Phi_p}^2(z)}$ are bounded in $[1/2,1]$, it follows that ${G'_p}(z)$ is also bounded in $[1/2,1]$.
Let $M_0>0$, so that $$|{F'_p}(\tau)|\le M_0,\ \  |{G'_p}(\tau)|\le M_0, \ \ \tau\in(1/2,1].$$
Recall that $$h_p=\frac{1}{2} (i F_p+G_p), \ \ \ g_p=\frac{1}{2} (-i F_p+G_p).$$ Thus we get $$|Df_p(\tau)|=|g_p'(\tau)|+|h_p'(\tau)|\le 2M_0.$$

  Then we get\begin{equation}|Df(\tau)|=|{f}_\zeta(\tau)|+|{f}_{\bar \zeta}(\tau)|\le 2 L_0 \Phi_0, \ \  \rho\le \tau\le \rho_1 \vee \rho_2<\tau<1, \end{equation} where $\rho_1<\rho_2$ are certain positive constants and $$\Phi_0=\max_{z\in\mathbf{D}}|\Phi'(z)|.$$
  Since the real interval $[\rho,1]$ has not a special geometric character for $\A_\rho$, we get that \begin{equation}|Df(z)|=|{f}_\zeta(z)|+|{f}_{\bar \zeta}(z)|\le 2 \Phi_0 M_2,  z\in B_\rho(\rho_1,\rho_2), \end{equation}  where $ B_\rho(\rho_1,\rho_2) =\{z=\tau e^{is}:  \rho< \tau\le \rho_1 \vee \rho_2<\tau<1, s\in [0,2\pi)\}.$

Since $f\in \mathscr{C}^{\infty}(\A_\rho)$ we get $f\in  \mathscr{C}^{0,1}(\overline{\A_\rho})$ as claimed, and thus the proof of Lemma~\ref{lemaimpo} is finished.

\end{proof}
\subsection{The minimizer is $\mathscr{C}^{1,\alpha'}$ up to the boundary}\label{sectioq}
We continue to use the notation from Subsections~\ref{lipsub} and \ref{subsec}. The constant $C$ that appear in the proof is not the same and its value can vary from one to the another appearance, but it is global and the same for all points of $\partial\A_\rho$. Assume that $f=u+iv:\A_\rho\to \Omega$ is a diffeomorphic minimizer, where $\A_\rho=\{z: \rho<|z|<1\}$,  we need to show that is $\mathscr{C}^{1,\alpha'}(\overline{\A_\rho})$, provided that $\partial\Omega\in \mathscr{C}^{1,\alpha}$.
 We only need to prove that $f$ is $\mathscr{C}^{1,\alpha'}$ in small neighborhood of $p\in\partial\A_\rho$. We also work with $f_p=f\circ \Phi_p:\mathbf{D}\to f(A_p)$ instead of $f$, where $\Phi_p(1)=1$ as in the previous part of the paper. We will show that $f_p\in \mathscr{C}^{1,\alpha'}(D_p)$, where $D_p=\{z=re^{is+is_0}: 1/2\le r<1, s\in(-\epsilon, \epsilon)\}$, where $p/|p|=e^{is_0}$, and $\epsilon>0$ is a small enough positive constant valid for all points  $p\in\partial\A_\rho$.

Assume as before that  $p=1$ and $f(p)=0=q\in \partial \Omega$.
Recall that
$\Lambda=\Lambda_p=\{e^{i\theta}: |\theta|\le \epsilon\}$. We already proved that $f$ is Lipschitz continuous. We know as well that $\gamma(x)=(x,\phi(x))\in \mathscr{C}^{1,\alpha}$. Since $f(e^{i\theta})=\gamma(x(e^{i\theta}))$,  we have that  $x=x(e^{i\theta})$ is Lipschitz continuous.

Since $f$ is a diffeomorphism, there exists a non-decreasing continuous function $x:\Lambda\to \mathbf{R}$ so that $$f(e^{it})=u(e^{it})+i v(e^{it})=\gamma(x(e^{it})).$$

We can also assume that \begin{equation}\label{assume}\partial_t x(e^{it})\ge 0\end{equation} for almost every $t$, because $f$ is a restriction of a harmonic diffeomorphism between domains and by Proposition~\ref{newcara} it is monotone at the boundary.

It follows from \eqref{maybe} and the fact that $x$ is Lipschitz and $x(1)=0$ that $v$ is differentiable with respect to $\theta$ for $\theta=0$, i.e. in $1$  and

\begin{equation}\label{zerod}\partial_{\theta} v(1)=\partial_{\theta} v(e^{i\theta})|_{\theta=0}=0.\end{equation}
Therefore
\begin{equation}\label{begeq}\begin{split}|\partial_\theta v(e^{i\theta})-\partial_{\theta} v(1)|&=|\partial_\theta v(e^{i\theta})|\\&=|\phi'(x(e^{i\theta}))|\cdot |\partial_\theta x(e^{i\theta})|\\&\le C|x(e^{i\theta})|^\alpha\le C|\theta|^\alpha,\end{split}\end{equation} holds for a.e. $\theta$ in a certain interval.
Recall from \eqref{begeq0} that $v_\theta =\Re (z (g'-h'))$, so from  Lemma~\ref{privalov2} we conclude that
\begin{equation}\label{seco}
|(z(g'-h'))'(\tau)|\le C(1-\tau)^{\alpha-1}, \ \ \ 1/2\le \tau\le 1.
\end{equation}
We also recall that we defined  \begin{equation}\label{k12}k_1(z) = i ( g'(z) + h'(z)), \ \ k_2(z) = h'(z) -g'(z).\end{equation} In view of \eqref{kii} \begin{equation}\label{k3}k_3(z) =  \sqrt{4h'(z) g'(z)}=\sqrt{ 4c \frac{(\Phi_p'(z))^2}{\Phi_p^2(z)}}.\end{equation}
Then from \eqref{seco} we have that the following limit   $$k_2(1) = i\partial_{t}(h(e^{it})-g(e^{it}))|_{t=0}$$ exists. Moreover, it can be assumed that $$k_2(1) = i\partial_{t}(h(e^{it})-g(e^{it})),$$ because $$\lim_{\tau \to 1} h'(\tau e^{is})=-i e^{-is} \partial_{t}h(e^{it})|_{t=s}$$

$$\lim_{\tau \to 1} g'(\tau e^{is})=-i e^{-is} \partial_{t}g(e^{it})|_{t=s}$$
for almost every $s\in(-\pi,\pi)$. This follows from the F. Riesz theorem \cite[Theorem~1, p.~409]{G}, because $|h'|$ and $|g'|$ are bounded.
Moreover we have from \eqref{seco}
\begin{equation}\label{k1}|k_2(1)-k_2(\tau)|\le C(1-\tau)^\alpha,  \ \ \ 1/2\le \tau\le 1. \end{equation}

We conclude that $$k_3(1)=2\lim_{r\to 1}\sqrt{h'(r) g'(r)}$$ exists and
\begin{equation}\label{k2}|k_3(1)-k_3(\tau)|\le C|1-\tau|^\alpha,\ \ 1/2\le \tau<1. \end{equation}
Further  since $$x(e^{i\theta})=u(e^{i\theta})=\Re (f(e^{i\theta}))=\Re (g+h),$$ we get
\begin{equation}\label{posi}
\partial_\theta x(e^{i\theta})=\Re \left[ie^{i\theta}(g'(e^{i\theta})+h'(e^{i\theta}))\right]=\Re \left[e^{i\theta}k_1(e^{i\theta})\right] \ge 0.\end{equation}
 Then the following equality is important in our approach \begin{equation}\label{mini}k_1^2+k_2^2+k_3^2=0.\end{equation} We now proceed as J. C. C. Nitsche did in \cite{nit}.
So $k_1^2(\tau)=-k_2^2(\tau)-k_3^2(\tau).$

It follows that the following limit $$k_1^2(1):=\lim_{\tau\to 1} k_1^2(\tau)=-k_2^2(1)-k_3^2(1),$$ exists.
Therefore we get $$|k_1(1)^2-k_1(\tau)^2|=|k_2^2(1)-k_2^2(\tau)+k_3^2(1)-k_3^2(\tau)|.$$
Then from \eqref{k1} and \eqref{k2} we get
 \begin{equation}\label{kk}|k_1(\tau)^2-k_1(1)^2|\le C |\tau-1|^{\alpha}\equiv\varepsilon, \ \  \ 1/2\le \tau<1.\end{equation}
From \eqref{zerod}, in view of \eqref{begeq0}, we get  \begin{equation}\label{rek2}\Re (k_2(1))=0.  \end{equation}
Further, from \eqref{mini}, we have \begin{equation}\label{iner}\Re (k_1(1))\Im (k_1(1))+\Re (k_2(1))\Im (k_2(1))+\Re (k_3(1))\Im (k_3(1))=0\end{equation} and
\begin{equation}\label{canot}\begin{split}\Re^2 (k_1(1))&+\Re^2 (k_2(1))+\Re^2 (k_3(1))\\&=\Im^2 (k_1(1))+\Im^2 (k_2(1))+\Im^2 (k_3(1)).\end{split}\end{equation}
Notice the following important fact, the relations \eqref{iner} and \eqref{canot} make sense for almost every $p\in\partial \A_\rho$. Namely $k_1(z) = P[k_i|_{\mathbf{T}}](z),$ $i=1,2,3$. We assume that $p=1$ is one of such points.
From Lemma~\ref{male} it follows that $k_3(1)$ is a real or an imaginary number. Therefore  we  have $\Re (k_3(1))\Im (k_3(1))=0$. Thus

\begin{equation}\label{deim}\Re (k_1(1))\Im (k_1(1))=0.\end{equation}

Now we divide the proof into two cases, and remember that the case $c=0$ coincides with the the case when the minimizer is a conformal biholomorphism:

(1)  {We first consider the case $c<0$ and put $\xi=1$}. In this case $\Re k_3(1)=0$.   But then cannot be $\Im (k_1(1))\neq 0$, because in that case $\Re (k_1(1))=0$, and therefore by \eqref{canot}, we get $\Im^2 (k_1(1))+\Im^2 (k_2(1))+\Im^2 (k_3(1))= 0$. The conclusion is that $\Im (k_1(1))= 0$.
Observe also that \begin{equation}\label{cruci}k_1(1)=\sqrt{-\Im^2 (k_2(1))-4c}\ge 2\sqrt{-c}> 0.\end{equation}

(2) { Then we  consider the case $c>0$ and put $\xi=-i \mathrm{sign} \Im k_1(1) $ if $\Im k_1(1)\neq 0$ and $\xi=1$ for the case  $\Im k_1(1)= 0$. }

Then we apply the following  lemma for $w_1=   \xi k_1(\tau)$ and $w_2= \xi k_1(1)$ and for $\varepsilon$ defined in \eqref{kk}

\begin{lemma}\cite{nit}
Let $w_1=a+ib$ and $w_2=\omega$ be complex numbers satisfying the inequalities $\omega\ge 0$ and $|w_1^2-w_2^2|\le \varepsilon$ for some $\varepsilon>0$. Then either $|w_1-w_2|\le 3\sqrt{\varepsilon}$ or $|w_1|\ge \sqrt{\varepsilon}$ and $a<0$, $\omega>0$.
\end{lemma}
Then as in \cite{nit} we  get
\begin{equation}\label{k11}
|k_1(\tau)-k_1(1)|\le C (1-\tau)^{\alpha/2}, \ \   1/2\le \tau<1.
\end{equation}
Recall now that \begin{equation}\label{k1k2gf}k_1(z) = G'_p(z)\ \ \ \ k_2(z) = F_p'(z).\end{equation} Then from \eqref{locglob} and \eqref{locglob1} we get \begin{equation}\label{phik}\Phi_p'(F_1(z))F'_p(z)=\Phi_1'(F_p(z))F'_1(z)\end{equation}
and
\begin{equation}\label{phig}\Phi_p'(G_1(z))G'_p(z)=\Phi_1'(G_p(z))G'_1(z)\end{equation}
for $z\in A_p\cap A_1$.
Therefore from \eqref{k1}, \eqref{k2} and \eqref{k11}, in a small $\epsilon-$neighborhood  of $p=1$,  we get the inequalities
\begin{equation}\label{k123}
|k_j(\tau e^{it})-k_j(e^{it})|\le C (1-\tau)^{\alpha/2}, \ \   1/2\le \tau<1, \  j=1,2
\end{equation}
for almost every  $t\in (-\epsilon,\epsilon)$.

Further as in \cite[p.~325-326]{nit} we obtain that $$|k_j(e^{it})-k_j(e^{is})|\le C|s-t|^{\frac{\alpha}{\alpha+2}}, j=1,2$$ and almost every $t,s\in(-\epsilon,\epsilon)$. The function $k_3$ has the same behavior a priori.
From this it follows that \begin{equation}\label{kjj}k_j\in \mathscr{C}^{0,{\frac{\alpha}{\alpha+2}}}(\overline{D_{p,\epsilon}}),\ \ \  j=1,2,3.\end{equation}

\emph{This concludes the case $c\ge 0$.}

\emph{Now we continue to prove the case $c\le 0$.}
This case we use of the parameterization
\begin{equation}\label{lift10} \varphi(z) =\left(\Re f(z) ,\Im f(z), 2 \sqrt{-c}\log\frac{1}{|z|}\right).\end{equation} By using this we get that $\varphi(\A_\rho)=\Sigma$ is a doubly connected  minimal surface bounded by two Jordan curves $$\Upsilon_1=\{(x,y, -2 \sqrt{-c}\log {r}), (x,y)\in \Gamma_1\}$$ and
$$\Upsilon_2=\{(x,y, -2 \sqrt{-c}\log {R}), (x,y)\in \Gamma_2\},$$ where $\Gamma_1$ and $\Gamma_2$ are the inner and outer boundaries of $\Omega$. Moreover $\partial\Omega\in \mathscr{C}^{1,\alpha}$ if and only if $\partial\Sigma\in \mathscr{C}^{1,\alpha}$.

This time Proposition~\ref{equation} will imply the result.

From Proposition~\ref{equation} we obtain that
$f\in \mathscr{C}^{1,\alpha}(\overline{\Phi_p^{-1}(A_p)})$, where $A_p$ is a small neighborhood of a fixed point $p$.  Notice that $A_p = p/|p| A_1$ or $A_p = p/|p| A_\rho$, where $A_1$ and $A_\rho$ are fixed domains, whose boundary contains a Jordan arc $\Lambda$, whose interior contains $1$ respectively $\rho$. Since $\A_\rho=\Phi_p^{-1}(A_p)$, we get that $f\in  \mathscr{C}^{1,\alpha}(\overline{\A_\rho})$ as claimed.

Thus we have finished the proof of Theorem~\ref{maine0}.
\section{Proof of Theorem~\ref{mainen}}\label{vised}
Having proved Theorem~\ref{maine0} this proof is a simple matter. We can reformulate a similar statement to Proposition~\ref{equation}, which is valid for higher-differentiability of the the mapping, and then use the harmonic mapping $\varphi$ defined  \eqref{lift10}, in order to get that $\varphi\in \mathscr{C}^{n,\alpha}(\A_\rho)$, provided that $\partial\Omega\in \mathscr{C}^{n,\alpha}(\A_\rho)$, which is equivalent with the condition $\partial\Sigma\in \mathscr{C}^{n,\alpha}(\A_\rho).$

\section{Concluding remark}
 We expect  that the following statement is true:
 \begin{conjecture}
 Assume that  $f:D\to \Omega$ is a energy minimal diffeomorphism of the energy between two domains with  $\mathscr{C}^{1,\alpha}$ boundaries. If $\mathrm{Mod}(D)\le\mathrm{Mod}(\Omega)$
 then the diffeomorphic minimizer of Dirichlet energy, which is shown to have a $\mathscr{C}^{1,\alpha'}$  extension up to the boundary is diffeomorphic on the boundary also and the extension is $\mathscr{C}^{1,\alpha}$ .
\end{conjecture}
This conjecture is motivated by the existing result described in Proposition~\ref{q4} and the  example presented in \eqref{nits}
 of the unique minimizer (up to the rotation) of Dirichlet energy between annuli $\A_r$ and $\A_R$, that maps the outer boundary onto the outer boundary (see \cite{AIM} for details). The mapping is a a diffeomorphism  between $\overline{\A_r}$ and $\overline{\A_R}$, provided that
\begin{equation}\label{jjcn}R< \frac{2r}{1+r^2}.\end{equation} If $R= \frac{2r}{1+r^2},$ and $0<r<1$, then the mapping
$$w(z)=\frac{r^2+|z|^2}{\bar z(1+r^2)}$$ is a harmonic minimizer (see \cite{AIM}) of the Euclidean energy of
mappings between $\mathbb{A}(r,1)$ and $\mathbb{A}(\frac{2r}{1+r^2},1)$, however $$|w_z|=|w_{\bar z}|=\frac{1}{1+r^2}$$ for $|z|=r$, and so $w$ is not
bi-Lipschitz.

Note that \eqref{jjcn} is satisfied provided that $\Mod \A_r\le \Mod \A_R$. The inequality \eqref{jjcn} (with $\le$ instead of $<$) is necessary and sufficient for the existence of a harmonic diffeomorphism between $\A_r$ and $\A_R$ a conjecture raised by J. C. C. Nitsche in \cite{Nitsche} and proved by Iwaniec,  Kovalev and  Onninen in \cite{IKO2},  after some partial results given by Lyzzaik \cite{L}, Weitsman \cite{W} and Kalaj
\cite{Ka}. If $$R> \frac{2r}{1+r^2},$$ then the minimizer of Dirichlet energy throughout the deformations $\mathcal{D}(\A_r,\A_R)$ is not a diffeomorphism ( see \cite{AIM} and \cite[Example~1.2]{cris}).

 We want to refer to one more interesting behavior that minimizers of energy share with conformal mappings. Namely, if $f$ is a diffeomorphic minimizer of Dirichlet energy between the domains $\A_\rho$ and $\Omega(\Gamma,\Gamma_1)$ so that $\Gamma$ and $\Gamma_1$ are convex, then $f(t\mathbf{T})$ is convex for $t\in (\rho,1)$ \cite{koh2}. Further if $\Gamma$ and $\Gamma_1$ are circles, then $f(t\mathbf{T})$ is a circle \cite{koh}.

\section*{Appendix-Proof of Lemma~\ref{newle}}

For $a\in \mathbf C$ and $r>0$,
put $D(a,r):=\{z:|z-a|<r\}$  and  define  $\Delta_r=\Delta_r(z_0) =
\mathbf D\cap D(z_0,r)$. Denote by $k_\tau $  the circular arc whose
trace is $ \{\zeta \in \mathbf D : | \zeta - z_0| = \tau \}$.

\begin{lemma}[The length-area principle]\label{newle0} \cite{kalmat}
Assume that   $f$ is a  $(K,K')-$ q.c. on  $\Delta_r$, $0< r <r_0\le 1$,
$z_0\in \mathbf T$ . Then

\begin{equation}\label{fr}F(r):= \int_0^r  \frac{l_\tau^2}{\tau} d\tau \leq \pi K  A(r) +
\frac{\pi}{2} K' r^2\,,
\end{equation}
where $l_\tau=|f(k_\tau)|$ denote the length of $f(k_\tau)$ and
$A(r)$ is the area of $f(\Delta_r)$.
\end{lemma}
\begin{proof}[Proof of Lemma~\ref{newle}]
Let $\Phi$ be a conformal mapping of $\Omega(\Gamma)$ onto the unit disk
, where $\Omega(\Gamma)$ is the Jordan domain bounded by $\Gamma$, so that $\Phi(f(1))=1$, $\Phi(f(e^{\pm i\frac{2\pi}{3}}))=e^{\pm i\frac{2\pi}{3}}.$

Then $\Phi\circ f$ is a normalized $(K_1,K_1')$ quasiconformal mapping near $\mathbf{T}\subset \partial \A_\rho$.
For $a\in \mathbf C$ and $r>0$, put $D(a,r):=\{z:|z-a|<r\}$.  Since $\Phi$ is a diffeomorphism near $\mathbf{T}$, the inequality \eqref{enjte} will be proved for $f$ if we prove it for $\Phi\circ f$.

 It is clear that if $z_0\in \mathbf T$, then, because of normalization, $f(\mathbf T\cap
\overline{D(z_0,1)})$ has common points with at most two of three
arcs ${w_0w_1}$, $w_1w_2$ and $w_2w_0$. (Here $w_0$, $w_1$, $w_2\in
\Gamma$ divide $\Gamma$ into three arcs with the same length such
that $f(1)=w_0$, $f(e^{2\pi i/3})=w_1$, $f(e^{4\pi i/3})=w_2$, and
$\mathbf T\cap \overline{D(z_0,1)}$ do not intersect at least one of
three arcs defined by $1$, $e^{2\pi i/3}$ and $e^{4\pi i/3}$).

 Let
$\kappa_\tau=\{t\in[0,2\pi]: z_0+\tau e^{it}\in k_\tau\}$. Let $l_\tau=|f(k_\tau)|$ denotes the length of $f(k_\tau)$. Let
$\Gamma_\tau:= f(\mathbf T\cap D(z_0,\tau))$ and let $|\Gamma_\tau|$
be its length. Assume $w$ and $w'$ are the endpoints of
$\Gamma_\tau$, i.e. of $f(k_\tau)$. Then $|\Gamma_\tau| =
d_\Gamma(w,w')$ or $|\Gamma_\tau| = |\Gamma| - d_\Gamma(w,w')$. If
the first case holds, then since $\Gamma$ enjoys the $B-$chord-arc
condition, it follows $|\Gamma_\tau|\le B|w-w'|\le Bl_\tau$.
Consider now the last case. Let $\Gamma_\tau' = \Gamma\setminus
\Gamma_\tau$. Then $\Gamma_\tau'$ contains one of the arcs
${w_0w_1}$, $w_1w_2$, $w_2w_0$. Thus $|\Gamma_\tau|\le
2|\Gamma_\tau'|$, and therefore $$|\Gamma_\tau|\le 2Bl_\tau.$$

   Using the first part of the proof, it
follows that the length of boundary arc $\Gamma_r$ of $f(\Delta_r)$
does not exceed $2Bl_r$ which, according to the fact that $\partial
f(\Delta_r)=\Gamma_r \cup f(k_r)$, implies
\begin{equation}\label{nevoja}|\partial f(\Delta_r)| \le l_r+2Bl_r.\end{equation} Therefore, by
the isoperimetric inequality $$ A(r) \le \frac{|\partial
f(\Delta_r)|^2}{4\pi}\le \frac{(l_r +2Bl_r)^2}{4\pi} =
l^2_r\frac{(1+2B)^2}{4\pi}.$$ Employing now \eqref{fr} we obtain $$F(r):=\int_0^r\frac{l^2_\tau}{\tau}d\tau \le
Kl^2_r\frac{(1+2B)^2}{4}+\frac{\pi K'}{2}r^2.$$ Observe that for
$0<r\le 1-\rho$ there holds $rF'(r) = l^2_r$. Thus $$F(r)\le K rF'(r)
\frac{(1+2B)^2}{4}+\frac{\pi K'}{2}r^2.$$ Let $G$ be the solution of
the equation $$G(r)= K rG'(r) \frac{(1+2B)^2}{4}+\frac{\pi
K'}{2}r^2, \  \ G(0)=0,$$ defined by $$G(r) = \frac{\frac{\pi
K'}{2}}{K\frac{(1+2B)^2}{4}+1}r^2=\frac{2{\pi
K'}}{K{(1+2B)^2}+4}r^2.$$ It follows that for $$\beta =
\frac{2}{K(1+2B)^2}$$ there holds $$\frac{d}{dr}\log
([F(r)-G(r)]\cdot r^{-2\beta})\ge 0,$$ i.e. the function
$[F(r)-G(r)]\cdot r^{-2\beta}$ is increasing. This yields

\[\begin{split}[F(r)-G(r)]&\le [F(1-\rho)-G(1-\rho)](r/(1-\rho))^{2\beta}\\&\le C(K,K',B, \rho, f)
r^{2\beta}.\end{split}\] Now for every $r\le 1-\rho$ there exists an
$r_1\in[r/\sqrt 2,r]$ such that
$$F(r)=\int_0^r\frac{l^2_\tau}{\tau}d\tau\ge \int_{r/\sqrt
2}^r\frac{l^2_\tau}{\tau}d\tau = l^2_{r_1}\log \sqrt 2. $$ Hence,
$$l^2_{r_1}\le \frac{C_1(K,K',B, \rho, f)}{\log 2}
r^{2\beta}.$$ If $z$ is a point with  $|z|\le 1$ and
$|z-z_0|=r/\sqrt 2$, then by \eqref{nevoja} $$|f(z)-f(z_0)|\le
(1+2B)l_{r_1}.$$ Therefore $$ |f(z)-f(z_0)|\le H|z-z_0|^\beta, $$
where $$H = H(K,K',B, \rho, f).$$

Now for $z_1,z_2\in \mathbf T$, then the arch $(z_1,z_2)$ can be divided into   $Q=Q(\rho)$ equal arcs by points $w_0,\dots, w_Q$, so that $|w_i-w_{i+1}|\le \frac{1-\rho}{\sqrt{2}}.$ Then we get the inequality $$|f(z_1)-f(z_2)|\le \sum_{j=1}^Q |f(w_j)-f(w_{j-1})|\le Q H |w_1-w_2|^\beta \le Q H|z_1-z_2|^\beta.$$ Thus
\begin{equation}\label{there}|f(z_1)-f(z_2)|\le
L(K,K',B, \rho, f)|z_1-z_2|^\beta.\end{equation}
In order to deal with the inner boundary, we take the composition $$F(z) = \frac{1}{f(\rho/z)-a},$$ which maps the annulus $\A_\rho$ into $\Omega'=\{1/(z-a): z\in \Omega\}$. Here $a$ is a point inside the inner Jordan curve. Then $\Omega'=\Omega'(\Gamma', \Gamma_1')$ is a doubly connected domain with $\mathscr{C}^{1,\alpha}$ boundary.

 Now we construct a conformal mapping $\Phi_1$ between the domain $\Omega(\Gamma')$ and the unit disk and repeat the previous case in order to get that the inequality \eqref{enjte} does hold in both boundary components.

 \end{proof}

\section*{Acknowledgement}
We would like to thank the anonymous  referee for a large number of remarks that helped to improve this paper. His/her idea is used to shorten the proof of the case $c<0$.

\end{document}